\documentclass[preprint]{elsarticle}
\usepackage{amssymb,amsmath,amsthm}

\newtheorem{corollary}{Corollary}
\newtheorem{proposition}{Proposition}
\newtheorem{theorem}{Theorem}
\newtheorem{lemma}{Lemma}

\numberwithin{equation}{section}

\begin{document}

\begin{frontmatter}
\title{Infinite products involving Dirichlet characters 
and cyclotomic polynomials}

\author{Karl Dilcher}
\address{Department of Mathematics and Statistics\\
 Dalhousie University\\
         Halifax, Nova Scotia, B3H 4R2, Canada}
\ead{dilcher@mathstat.dal.ca}
\author{Christophe Vignat \corref{cor1}}
\address{LSS-Supelec, Universit\'e Paris-Sud, Orsay, France\\ Department of
Mathematics, Tulane University, New Orleans, LA 70118, USA}
\ead{cvignat@tulane.edu}
\cortext[cor1]{Corresponding author}
\begin{abstract}
Using some basic properties of the gamma function, we evaluate a simple class
of infinite products involving Dirichlet characters as a finite product of
gamma functions and, in the case of odd characters, as a finite product of
sines. As a consequence we obtain evaluations of certain multiple $L$-series.
In the final part of this paper we derive expressions for infinite products
of cyclotomic polynomials, again as finite products of gamma or of sine 
functions.
\end{abstract}


\setcounter{equation}{0}

\end{frontmatter}


\section{Introduction}

Infinite products of various types play an important role in different areas of
mathematics, especially in analysis, number theory, and the theory of special
functions. On the one hand there are very general results, such as the 
following important theorem: 

{\it Let $a_1, a_2, a_3,\ldots$ be a sequence of real 
or complex numbers such that none of the $1+a_j$ vanishes. Then the infinite 
product, respectively the infinite series
\begin{equation}\label{1.1}
\prod_{k=1}^\infty\left(1+a_k\right)\qquad\hbox{and}\qquad \sum_{k=1}^\infty a_k
\end{equation}
either both converge, or both diverge.} (See, e.g., \cite[Sect.~2.7]{WW}, or
\cite[p.~104ff.]{Br} for a different proof in the case of real sequences.)

Divergence of an infinite product includes ``divergence to zero". An example
for this case is $a_k=\frac{-1}{k+1}$, i.e., a divergent harmonic series on the
right of \eqref{1.1}; this then corresponds to an infinite product on the left
whose $N$th partial product is easily seen to be $\frac{1}{N+1}$.

On the other hand, numerous specific infinite products can be found throughout
the literature, including compendia such as \cite[Sect.~89]{Ha}, 
\cite[Sect.~6.2]{PrE}, or the online resources \cite{We1} 
or \cite{Di}. While we just saw that the infinite product 
$\prod_{k=2}^\infty(1-\frac{1}{k})$ diverges, some well-known convergent 
variants, along with their values, are
\begin{equation}\label{1.2}
\prod_{k=2}^\infty\left(1-\frac{1}{k^2}\right)=\frac{1}{2},\qquad
\prod_{k=1}^\infty\left(1-\frac{(-1)^k}{2k+1}\right)=\frac{\pi}{4}\sqrt{2}.
\end{equation}
Also related to this is the Weierstrass factorization theorem which gives,
for example,
\begin{equation}\label{1.3}
\prod_{\substack{k=-\infty\\k\neq 0}}^\infty\left(1-\frac{z}{k}\right)e^{z/k}
=\frac{\sin(\pi z)}{\pi z},\qquad
ze^{\gamma z}\prod_{k=1}^\infty\left(1+\frac{z}{k}\right)e^{-z/k}
=\frac{1}{\Gamma(z)},
\end{equation}
where $\Gamma(z)$ is Euler's gamma function, and $\gamma$ is the 
Euler-Mascheroni constant. The two well-known identities in \eqref{1.3} are 
examples of infinite products representing functions, and in fact, the two
products in \eqref{1.2} are just special cases of such representations (see,
e.g., (89.5.11) and (89.4.11), respectively, in \cite{Ha}). The first identity
in \eqref{1.2} actually follows immediately from the first identity in 
\eqref{1.3}.

The second identity in \eqref{1.3} is a first indication of the fact that 
many infinite products are closely related to the gamma function. Indeed, a 
second general result, after \eqref{1.1}, states that a convergent infinite
product of a rational function in the index $k$ can always be written as a 
product or quotient of finitely many values of the gamma function. The expansion
on the right of \eqref{1.3} is an ingredient in the proof of this fact; see
\cite[Sect.~12.13]{WW}, or \cite{CS} for further comments and examples.

The first part of this paper will be concerned with a generalization of the 
right-hand identity
in \eqref{1.2}. In fact, this identity can be written in terms of a certain
Dirichlet character. This will be the starting point for considering such sums
for arbitrary real or complex nontrivial Dirichlet characters. More generally,
we can also consider corresponding functions of a complex variable.

Our main results, along with some consequences and examples, will be stated 
in Section~2, and in Section~3 we will prove these results. In Section~4 we 
use them to obtain evaluations for certain multiple $L$-series, and in 
Section~5 for analogous multiple $L$-star series.
In Section~6 we use the left-hand identity of \eqref{1.2} as a point of 
departure and obtain evaluations for infinite products of cyclotomic 
polynomials. We finish this paper with some additional remarks in Section~6.

\section{The main results}

Let $\chi$ be the unique nontrivial Dirichlet character modulo 4, i.e., the 
periodic function of period 4 defined by $\chi(1)=1, \chi(3)=-1$, and 
$\chi(0)=\chi(2)=0$. Then we can rewrite the second identity in \eqref{1.2}  as 
\begin{equation}\label{2.1}
\prod_{k=2}^\infty\left(1-\frac{\chi(k)}{k}\right)=\frac{\pi}{4}\sqrt{2},
\end{equation}
and as a function of the real or complex variable $z$ as
\begin{equation}\label{2.2}
\prod_{k=2}^\infty\left(1-\chi(k)\frac{z}{k}\right)
=\frac{\sqrt{2}}{1-z}\sin{\frac{(1-z)\pi}{4}};
\end{equation}
see, e.g., \cite[(89.4.11)]{Ha} or \cite[p.~224]{Br}. The identity \eqref{2.1}
follows immediately from \eqref{2.2} by taking the limit as $z\rightarrow 1$.

Our main result is the following generalization of \eqref{2.2}. For properties
of Dirichlet characters, the reader may consult, e.g., \cite{Ap}.

\begin{theorem}
Let $\chi$ be a primitive nonprincipal Dirichlet character with conductor 
$q>1$. Then
\begin{equation}\label{2.3}
\prod_{k=2}^\infty\left(1-\chi(k)\frac{z}{k}\right)
=\frac{(2\pi)^{\varphi(q)/2}}{(1-z)\epsilon(q)}
\cdot\prod_{\substack{j=1\\(j,q)=1}}^{q-1}\frac{1}{\Gamma\left(\frac{j-\chi(j)z}{q}\right)},
\end{equation}
where $\varphi(q)$ is Euler's totient function, $(j,q)$ denotes the gcd, and
$\epsilon(q)$ is defined by
\[
\epsilon(q)=\begin{cases}
\sqrt{p} &\hbox{when $q$ is a power of a prime $p$},\\
1 &\hbox{otherwise}.
\end{cases}
\]
\end{theorem}

In the case where $\chi$ is an odd character, i.e., $\chi(-1)=-\chi(1)=-1$,
we obtain the following result from \eqref{2.3}.

\begin{theorem}
Let $\chi$ be an odd primitive Dirichlet character with conductor $q>1$. 
Then
\begin{equation}\label{2.4}
\prod_{k=2}^\infty\left(1-\chi(k)\frac{z}{k}\right)
=\frac{2^{\varphi(q)/2}}{(1-z)\epsilon(q)}
\cdot\prod_{\substack{j=1\\(j,q)=1}}^{\lfloor\frac{q-1}{2}\rfloor}
\sin\left(\pi\frac{j-\chi(j)z}{q}\right).
\end{equation}
\end{theorem}

As an immediate consequence of these results we consider the case $z=1$. To 
deal with the removable singularity at $z=1$ on the right-hand sides of 
\eqref{2.3} and \eqref{2.4}, we isolate the factor belonging to $j=1$ in both
cases and use the limits
\begin{equation}\label{2.5}
\lim_{z\rightarrow 1}\left((1-z)\Gamma(\tfrac{1-z}{q})\right) = q,\qquad
\lim_{z\rightarrow 1}\frac{\sin(\frac{\pi}{q}(1-z))}{1-z}=\frac{\pi}{q}.
\end{equation}
Here the first limit in \eqref{2.5} comes from the well-known series expansion
$z\Gamma(z)=1-\gamma z+O(z^2)$. This proves the following corollary.

\begin{corollary}
Let $\chi$ be a primitive nonprincipal Dirichlet character with conductor
$q>1$. Then
\begin{equation}\label{2.6}
\prod_{k=2}^\infty\left(1-\frac{\chi(k)}{k}\right)
=\frac{(2\pi)^{\varphi(q)/2}}{q\epsilon(q)}
\cdot\prod_{\substack{j=2\\(j,q)=1}}^{q-1}\frac{1}{\Gamma\left(\frac{j-\chi(j)}{q}\right)},
\end{equation}
and when $\chi$ is odd, we have
\begin{equation}\label{2.7}
\prod_{k=2}^\infty\left(1-\frac{\chi(k)}{k}\right)
=\frac{\pi 2^{\varphi(q)/2}}{q\epsilon(q)}
\cdot\prod_{\substack{j=2\\(j,q)=1}}^{\lfloor\frac{q-1}{2}\rfloor}\sin\left(\pi\frac{j-\chi(j)}{q}\right).
\end{equation}
\end{corollary}

We conclude this section with several examples.

\medskip
\noindent
{\bf Example~1.} Let $q=3$. Then the only nonprincipal character is determined
by $\chi(1)=1$, $\chi(2)=-1$, and $\chi(3)=0$. With $\varphi(q)=2$ and 
$\epsilon(q)=3$, the identity \eqref{2.4} gives
\begin{equation}\label{2.7a}
\prod_{k=2}^\infty\left(1-\chi(k)\frac{z}{k}\right)
=\frac{2}{(1-z)\sqrt{3}}\cdot\sin(\tfrac{\pi}{3}(1-z)),
\end{equation}
and with $z=1$ and $z=\frac{1}{2}$ we get, respectively,
\[
\prod_{k=2}^\infty\left(1-\frac{\chi(k)}{k}\right)
=\frac{2\pi}{3\sqrt{3}},\qquad
\prod_{k=2}^\infty\left(1-\frac{\chi(k)}{2k}\right)
=\frac{2}{\sqrt{3}}
\]

\medskip
\noindent
{\bf Example~2.} When $q=4$, we have again a unique nonprincipal character, 
namely the one defined just before \eqref{2.1}. Since $\varphi(q)=2$ and
$\epsilon(q)=\sqrt{2}$, the identities \eqref{2.4} and \eqref{2.7} immediately
give \eqref{2.2} and \eqref{2.1}, respectively.

\medskip
\noindent
{\bf Example~3.} In the case $q=5$ there are $\varphi(q)-1=3$ nonprincipal
characters, with the unique real one given by $\chi(1)=\chi(4)=1$,
$\chi(2)=\chi(3)=-1$, and $\chi(5)=0$. Since $\chi$ is even, we can only use
\eqref{2.3} or \eqref{2.6}. Restricting ourselves to the case $z=1$, and 
noting that $\epsilon(q)=\sqrt{5}$, we find from \eqref{2.6} that 
\[
\prod_{k=2}^\infty\left(1-\frac{\chi(k)}{k}\right)
=\frac{4\pi^2}{5\sqrt{5}}\cdot\frac{1}{\Gamma(\tfrac{3}{5})^2\Gamma(\tfrac{4}{5})}.
\]

\medskip
\noindent
{\bf Example~4.} $q=5$ is also the smallest conductor that has nonreal 
characters. We choose the one (of two) that is given by $\chi(1)=1$, 
$\chi(2)=i$, $\chi(3)=-i$, $\chi(4)=-1$, and $\chi(5)=0$. Then \eqref{2.4} gives
\[
\prod_{k=2}^\infty\left(1-\chi(k)\frac{z}{k}\right)
=\frac{4}{(1-z)\sqrt{5}}\cdot\sin(\tfrac{\pi}{5}(1-z))\cdot\sin(\tfrac{\pi}{5}(2-iz)),
\]
and with \eqref{2.7} we have
\begin{align*}
\prod_{k=2}^\infty\left(1-\frac{\chi(k)}{k}\right)
&=\frac{4\pi}{5\sqrt{5}}\cdot\sin(\tfrac{\pi}{5}(2-i))\\
&=\frac{4\pi}{5\sqrt{5}}\left(\sin(\tfrac{2\pi}{5})\cosh(\tfrac{\pi}{5})
-i\cos(\tfrac{2\pi}{5})\sinh(\tfrac{\pi}{5})\right).
\end{align*}

\medskip
\noindent
{\bf Example~5.} The smallest conductor that is not a power of a prime is $q=6$,
with the unique nonprincipal character given by $\chi(1)=1$, $\chi(5)=-1$, and
all the other values (modulo 6) being 0. We have $\varphi(q)=2$ and
$\epsilon(q)=1$; hence \eqref{2.4} gives 
\begin{equation}\label{2.7b}
\prod_{k=2}^\infty\left(1-\chi(k)\frac{z}{k}\right)
=\frac{2}{(1-z)}\cdot\sin(\tfrac{\pi}{6}(1-z)),
\end{equation}
with the special values 
\[
\prod_{k=2}^\infty\left(1-\frac{\chi(k)}{k}\right)
=\frac{\pi}{3},\qquad
\prod_{k=2}^\infty\left(1-\frac{\chi(k)}{2k}\right)
= 2\sqrt{2-\sqrt{3}}.
\]

\medskip
\noindent
{\bf Remarks.} (1) In \cite{Ya}, Proposition~3.1, Yamasaki proved the following
variant of our Theorem~2: For an odd Dirichlet character modulo $q\geq 3$, we
have
\begin{equation}\label{2.8}
\prod_{k=1}^\infty\left(1-\chi(k)\frac{z}{k}\right)
=\frac{2^{(q-1)/2}}{\sqrt{q}}
\cdot\prod_{j=1}^{\lfloor\frac{q-1}{2}\rfloor}
\sin\left(\pi\frac{j-\chi(j)z}{q}\right).
\end{equation}
He also obtained an extension of \eqref{2.8} for higher powers of $z/k$ in the
product on the left of \eqref{2.8}, as well as analogues for the trivial 
character and the principal character modulo~2.
The main difference between \eqref{2.8} and \eqref{2.4} lies in the fact that
the number of factors on the right of \eqref{2.4} is $\frac{1}{2}\varphi(q)$, 
which is fewer than in \eqref{2.8} for composite $q$.

(2) It will be clear from the proof in the next section that Theorem~1 holds for
arbitrary periodic functions $\chi$ with period $q$, satisfying 
$\chi(1)+\dots+\chi(q)=0$ and $\chi(j)=0$ when $\gcd(j,q)>1$. Theorem~2 holds
when, in addition, $\chi$ is an odd function. One could also obtain versions
of Theorems~1 and~2 in cases where $\chi(j)$ is not necessarily 0 when 
$\gcd(j,q)>1$.
 
\section{Proofs of Theorems~1 and~2}

The basis for the proof of Theorem~1 is the infinite product expansion for 
$1/\Gamma(z)$ in \eqref{1.3}. Replacing $z$ by $u+v$ and by $v$, dividing the
corresponding products, and slightly shifting the terms in the product, we get
\begin{equation}\label{3.1}
\frac{\Gamma(u)}{\Gamma(u+v)} =
e^{\gamma v}\prod_{k=0}^\infty\left(1+\frac{v}{u+k}\right)e^{-v/(k+1)}.
\end{equation}
This can also be found in \cite[p.~5]{EMOT} as identity (2). It will be the 
basis for the following lemma.

\begin{lemma}
Let $n\geq 1$ be an integer, $a, z_1,\ldots,z_n$ real or complex numbers with
$z_j\neq 0$ for $j=1,\ldots,n$, and let 
$f:\{1,2,\ldots,n\}\rightarrow\mathbb C$ be a function satisfying
$f(1)+\dots+f(n)=0$. Then
\begin{equation}\label{3.2}
\prod_{k=0}^\infty\prod_{j=1}^n\left(1-f(j)\frac{a}{z_j+k}\right)
=\prod_{j=1}^n\frac{\Gamma(z_j)}{\Gamma(z_j-f(j)a)}.
\end{equation}
\end{lemma}

\begin{proof}
for each $j=1,\ldots,n$ we replace $u$ by $z_j$ and $v$ by $-f(j)a$ in 
\eqref{3.1}. In taking the product of the resulting $n$ identities, we note that
\[
\prod_{j=1}^n e^{-\gamma f(j)} = e^{-\gamma(f(1)+\dots+f(n))} = e^0 = 1,
\]
and similarly for the second exponential term on the right of \eqref{3.1}.
Therefore the product in question immediately gives \eqref{3.2}. The operations
with the infinite products are legitimate since the product in \eqref{3.1} is
absolutely convergent.
\end{proof}

\noindent
{\bf Example~6.} Let $n=2$ and $f(1)=1, f(2)=-1$. Then \eqref{3.2} reduces to
\begin{equation}\label{3.3}
\prod_{k=0}^\infty\left(1-\frac{a}{z_1+k}\right)\left(1+\frac{a}{z_2+k}\right)
=\frac{\Gamma(z_1)\Gamma(z_2)}{\Gamma(z_1-a)\Gamma(z_2+a)}.
\end{equation}
This is, essentially, identity (4) in \cite[p.~5]{EMOT} or (89.9.2) in 
\cite{Ha}. Therefore \eqref{3.2} can be seen as a generalization of this
known identity.

\medskip
\noindent
{\bf Remark.} The condition $z_j\neq 0$ for $j=1,\ldots, n$ in Lemma~1 can be 
slightly relaxed by allowing $z_j=0$ when $f(j)=0$. However, in this case $z_j$
would be a redundant parameter anyway. On the other hand, the case 
$z_j-f(j)a=0$ is allowable since the identity on the right of \eqref{1.3} holds
for all $z\in\mathbb C$. In this case, both sides of \eqref{3.2} vanish.

\medskip
Our next lemma, due to Chamberland and Straub \cite[Theorem~5.1]{CS}, is an
extension of the well-known identity 
\begin{equation}\label{3.4}
\prod_{j=1}^{n-1}\Gamma\left(\frac{j}{n}\right) = \frac{(2\pi)^{(n-1)/2}}{\sqrt{n}},
\end{equation}
which is itself a consequence of Gauss's multiplication formula; see, e.g.,
Section 5.5(iii) in \cite{DLMF}.

\begin{lemma}[Chamberland and Straub]
For any integer $n\geq 2$ and prime $p$ we have
\begin{equation}\label{3.5}
\prod_{\substack{j=1\\(j,n)=1}}^{n-1}\Gamma\left(\frac{j}{n}\right) 
=\begin{cases}
(2\pi)^{\varphi(n)/2} &\hbox{if $n$ is not a prime power},\\ 
\frac{1}{\sqrt{p}}(2\pi)^{\varphi(n)/2} &\hbox{if $n=p^\nu, \nu\geq 1$.}
\end{cases}
\end{equation}
\end{lemma}

If $n$ is itself a prime, we see that \eqref{3.5} reduces to \eqref{3.4}. We are
now ready to prove Theorem~1.

\begin{proof}[Proof of Theorem~1]
For any nonprincipal character $\chi$ with conductor $q$ we have 
$\chi(1)+\chi(2)+\dots+\chi(q)=0$. We can therefore use Lemma~1 with $n=q$,
$f(j)=\chi(j)$, $a=z/q$, and $z_j=j/q$. Then the left-hand side of \eqref{3.2} 
becomes
\begin{equation}\label{3.6}
\prod_{k=0}^\infty\prod_{j=1}^q\left(1-\chi(j)\frac{z}{kq+j}\right)
= (1-z)\prod_{k=2}^\infty\left(1-\chi(k)\frac{z}{k}\right),
\end{equation}
where we have used the fact that $\chi(1)=1$ for all Dirichlet characters.

To deal with the right-hand side of \eqref{3.2}, we first note that for all
Dirichlet characters we have $\chi(j)=0$ whenever $j$ and $q$ have a common
factor. Hence the product can be reduced to the one shown on the right of
\eqref{2.3}. Finally we note that the product of the numerators in \eqref{3.2}
can be evaluated with Lemma~2. Combining this with \eqref{3.4} and \eqref{3.2},
we immediately get \eqref{2.3}.
\end{proof}

\begin{proof}[Proof of Theorem~2]
This proof depends on Theorem~1 and another well-known and important property
of the gamma function, namely the reflection formula
\begin{equation}\label{3.7}
\Gamma(z)\Gamma(1-z) = \frac{\pi}{\sin(\pi z)},\quad z\neq 0,\pm 1,\pm 2,\ldots
\end{equation}
If $\chi$ is an odd character, then $\chi(q-j)=-\chi(j)$ for all integers $j$,
and for $j=1,2,\ldots,\lfloor(q-1)/2\rfloor$ we have with \eqref{3.7},
\begin{align}
\Gamma\left(\frac{j-\chi(j)z}{q}\right)&\Gamma\left(\frac{q-j-\chi(q-j)z}{q}\right)\nonumber \\
&=\Gamma\left(\frac{j-\chi(j)z}{q}\right)\Gamma\left(1-\frac{j-\chi(j)z}{q}\right)
= \frac{\pi}{\sin\left(\pi\frac{j-\chi(j)z}{q}\right)}.\label{3.8}
\end{align}
We note that for even $q>2$, the ``middle term" $q/2$ is obviously a nontrivial
divisor of $q$ and is therefore not counted. Hence \eqref{2.3} gives rise to 
$\varphi(q)/2$ factors of the type \eqref{3.8}, so the powers of $\pi$ in 
\eqref{2.3} cancel and we get \eqref{2.4}. This completes the proof of 
Theorem~2.
\end{proof}

\medskip
\noindent
{\bf Remarks.} (1) For real primitive characters $\chi$ 
with prime conductors $q$, i.e., $\chi(j)$ is the Legendre symbol
$(\frac{j}{q})$, the two lemmas are not needed. In this case, Theorem~1 can be
proved with the help of the identities \eqref{3.3} and \eqref{3.4}.

(2) If we set $z=0$ in \eqref{2.4}, then we immediately get
\[
\prod_{\substack{j=1\\(j,q)=1}}^{\lfloor\frac{q-1}{2}\rfloor}
\sin(\tfrac{\pi j}{q}) = \frac{\epsilon(q)}{2^{\varphi(q)/2}},
\]
with $\epsilon(q)$ as in Theorem~1. This can be seen as a companion identity to
\eqref{3.5}. In fact, it can also be derived from \eqref{3.5} by using 
\eqref{3.7}. It extends a well-known identity (see, e.g., 
\cite[Entry 6.1.2.3, p.~753]{PrE}) in which the product is taken over the whole
range, i.e., including the terms for which $\gcd(j,q)\neq 1$.

\section{Values of single and multiple $L$-series}

In this section we use Theorems~1 and~2 to derive expressions for some special
values of single and multiple $L$-series. We require the well-know digamma
function (or psi function), defined by
\begin{equation}\label{4.1}
\psi(z) = \frac{\Gamma'(z)}{\Gamma(z)},\qquad z\neq 0, -1, -2,\ldots
\end{equation}
We begin with the easiest case.

\begin{corollary}
Let $\chi$ be a primitive nonprincipal character with conductor $q>1$. Then
\begin{equation}\label{4.2}
\sum_{k=1}^\infty\frac{\chi(k)}{k}
=\frac{-1}{q}\sum_{j=1}^{q-1}\chi(j)\psi(\tfrac{j}{q}),
\end{equation}
and when $\chi$ is odd,
\begin{equation}\label{4.3}
\sum_{k=1}^\infty\frac{\chi(k)}{k}
=\frac{\pi}{q}\sum_{j=1}^{\lfloor\frac{q-1}{2}\rfloor}
\chi(j)\cot(\tfrac{\pi j}{q}).
\end{equation}
\end{corollary}

This result is well-known; for instance, \eqref{4.2} can be found in 
\cite[p.~168]{Coh} as Proposition~10.2.5(4).

\begin{proof}[Proof of Corollary~2]
After multiplying both sides of \eqref{2.3} and \eqref{2.4} by $1-z$ and
expanding, we see that the coefficient of $z$ is the negative of the left-hand
sides of \eqref{4.2} and \eqref{4.3}. 

Next we evaluate the derivative at $z=0$ of the finite product on the right 
of \eqref{2.3}, obtaining
\[
\prod_{\substack{j=1\\(j,q)=1}}^{q-1}\Gamma(\tfrac{j}{q})^{-1}
\sum_{\substack{j=1\\(j,q)=1}}^{q-1}\Gamma(\tfrac{j}{q})
\cdot\frac{-\Gamma'(\tfrac{j}{q})}{\Gamma(\tfrac{j}{q})^2}
\cdot\frac{-\chi(j)}{q}
=\frac{\epsilon(q)}{(2\pi)^{\varphi(q)/2}q}
\sum_{\substack{j=1\\(j,q)=1}}^{q-1}\psi(\tfrac{j}{q})\chi(j),
\]
where we have used \eqref{3.5} and the definition \eqref{4.1}. This, together
with \eqref{2.3}, gives the desired identity \eqref{4.2}. 

In an analogous way we obtain \eqref{4.3} from \eqref{2.7}. Alternatively, and
more easily, \eqref{4.3} follows immediately from \eqref{4.2} by way of the 
well-known identity
\begin{equation}\label{4.3a}
\psi(z)-\psi(1-z) = -\pi\cot(\pi z),\quad z\neq 0,\pm 1,\pm 2,\ldots,
\end{equation}
which can be obtained by differentiating \eqref{3.7} and using \eqref{4.1}.
\end{proof}

\medskip
\noindent
{\bf Example~7.} With the usual notation of $\chi_3$ and $\chi_{-4}$ for the
nonprincipal characters with $q=3$ and $q=4$, respectively, \eqref{4.3}
immediately gives 
\[
\sum_{k=1}^\infty\frac{\chi_3}{k} = \frac{\pi}{3\sqrt{3}},\qquad
\sum_{k=1}^\infty\frac{\chi_{-4}}{k} = \frac{\pi}{4}.
\]
The second of these identities is the well-known formula of Gregory and 
Leibniz.

\medskip
More generally, our results in Section~2 can be used to obtain some special
values for multiple $L$-functions. In a more general setting these are defined
by
\begin{equation}\label{4.3b}
L_n(s_1,\ldots,s_n\mid \chi_1,\ldots,\chi_n)
:=\sum_{1\leq k_1<\dots<k_n}\frac{\chi_1(k_1)}{k_1^{s_1}}\dots\frac{\chi_n(k_n)}{k_n^{s_n}},
\end{equation}
where $\chi_1,\ldots,\chi_n$ are Dirichlet characters with the same conductor
$q\geq 2$. If Re$(s_j)\geq 1$ for $j=1,2,\ldots,n-1$ and Re$(s_n)> 1$, then
this series is absolutely convergent and defines a holomorphic function in the
$n$ complex variables $s_1,\ldots,s_n$ in this region; it can also be 
meromorphically continued to all of ${\mathbb C}^n$ (see \cite{AI}).

Here we will only be interested in the values at $s_1=\dots=s_n=1$, in the
special case where $\chi_1=\dots=\chi_n=\chi$, i.e., we consider the multiple
series
\begin{equation}\label{4.4}
L_n(\chi)
:=\sum_{1\leq k_1<\dots<k_n}\frac{\chi(k_1)}{k_1}\dots\frac{\chi(k_n)}{k_n}.
\end{equation}
By expanding the infinite product, we obtain
\begin{equation}\label{4.4a}
\prod_{k=1}^\infty\left(1-\chi(k)\frac{z}{k}\right)
= 1 + \sum_{n=1}^\infty(-1)^nL_n(\chi)z^n.
\end{equation}
As a first application of this identity we use three examples of Theorem~2,
namely the cases of odd conductors such that there is only one sine factor on
the right of \eqref{2.4}; see Examples~1, 2, and~5, respectively.

\begin{corollary}
$(a)$ If $\chi$ is the nonprincipal character with $q=3$, then
\begin{equation}\label{4.4b}
L_{2n}(\chi) = \frac{(-1)^n}{(2n)!}\left(\frac{\pi}{3}\right)^{2n},\qquad
L_{2n+1}(\chi)=\frac{(-1)^n}{(2n+1)!\sqrt{3}}\left(\frac{\pi}{3}\right)^{2n+1}.
\end{equation}
$(b)$ If $\chi$ is the nonprincipal character with $q=4$, then
\begin{equation}\label{4.4c}
L_{2n}(\chi) = \frac{(-1)^n}{(2n)!}\left(\frac{\pi}{4}\right)^{2n},\qquad
L_{2n+1}(\chi)=\frac{(-1)^n}{(2n+1)!}\left(\frac{\pi}{4}\right)^{2n+1}.
\end{equation}
$(c)$ If $\chi$ is the nonprincipal character with $q=6$, then
\begin{equation}\label{4.4d}
L_{2n}(\chi) = \frac{(-1)^n}{(2n)!}\left(\frac{\pi}{6}\right)^{2n},\qquad
L_{2n+1}(\chi)=\frac{(-1)^n\sqrt{3}}{(2n+1)!}\left(\frac{\pi}{6}\right)^{2n+1}.
\end{equation}
\end{corollary}

\begin{proof}
We multiply both sides of the identities \eqref{2.7a}, \eqref{2.2}, and 
\eqref{2.7b} by $1-z$ and compute the Maclaurin expansions of the right-hand
sides. Upon equating coefficients of $z^n$ with the right-hand side of 
\eqref{4.4a}, we obtain the six identities in question.
\end{proof}

\noindent
{\bf Example~8.} If we take $n=0$ in the second identities in \eqref{4.4b} and
\eqref{4.4c}, we recover the two identities in Example~7.

\medskip
To simplify notation in what follows, we set
\begin{equation}\label{4.5}
f_j(z) := \Gamma\left(\frac{j-\chi(j)z}{q}\right)^{-1},\qquad
1\leq j\leq q-1,\; \gcd(j,q)=1.
\end{equation}

Before giving an explicit identity for $L_n(\chi)$ in the general case, we 
state and prove an intermediate result. 

\begin{lemma} 
Let $\chi$ be a primitive nonprincipal Dirichlet character with conductor $q$.
Then for all integers $n\geq 1$ we have
\begin{equation}\label{4.6}
L_n(\chi) = (-1)^n\frac{(2\pi)^{\varphi(q)/2}}{\epsilon(q)}
{\sum}^*\prod_{j\in\Phi}\frac{1}{k_j!}f_j^{(k_j)}(0),
\end{equation}
with the index set $\Phi:=\{j\mid 1\leq j\leq q-1, \gcd(j,q)=1\}$, and where
the summation $\sum^*$ in \eqref{4.6} is over all $k_j (j\in\Phi)$ that sum to $n$.
\end{lemma}

\begin{proof}
As we did in the proof of Corollary~2, we multiply both sides of \eqref{2.3} by
$1-z$. With the aim of equating coefficients of $z^n$, we see that on the
left-hand side we have the expansion \eqref{4.4a}

For the right-hand side of \eqref{2.3}, we first note that by the Leibniz rule
(the general product rule) we have
\begin{equation}\label{4.8}
\frac{d^n}{dz^n}\prod_{j\in\Phi}\left.f_j(z)\right|_{z=0}
= {\sum}^*\frac{n!}{\prod_{j\in\Phi}k_j!}\prod_{j\in\Phi}f_j^{(k_j)}(0),
\end{equation}
where $\sum^*$ is the same as in the statement of the lemma. Finally, using a
Maclaurin expansion of the product on the right of \eqref{2.3}, canceling the
term $n!$, and equating coefficients of $z^n$ with \eqref{4.4a}, we obtain the
desired identity \eqref{4.6}.
\end{proof}

Next we require a well-known and important combinatorial object, namely the
{\it partial\/} (or {\it incomplete\/}) {\it exponential Bell polynomial} 
defined by 
\begin{equation}\label{4.9}
B_{n,k}(x_1,x_2,\ldots,x_{n-k+1})
=\sum\frac{n!}{j_1!\dots j_{n-k+1}!}
\left(\frac{x_1}{1!}\right)^{j_1}\cdots
\left( \frac{x_{n-k+1}}{(n-k+1)!} \right)^{j_{n-k+1}},
\end{equation}
where the summation is taken over all nonnegative integers 
$j_1,j_2,\ldots,j_{n-k+1}$ satisfying the two conditions
\begin{align*}
j_1+2j_2+\dots+(n-k+1)j_{n-k+1}&=k,\\
j_1+j_2+\dots+j_{n-k+1}&=n;
\end{align*}
see, e.g., \cite[Sect.~3.3]{Com}, and \cite[p.~301]{Com} for a table.

\medskip
\noindent
{\bf Example~9.} Three general cases of \eqref{4.9} are easy to obtain: 
(i) When $k=0$, then 
the sum in \eqref{4.9} is empty; (ii) when $k=1$, then $j_n=1$ provides the 
only summand; (iii) when $k=n$, then $j_1=n$ provides the only summand. Hence 
for all $n\geq1$ we have
\[
B_{n,0}(x_1,x_2,\ldots,x_{n+1})=0,\qquad B_{n,1}(x_1,x_2,\ldots,x_{n})=x_n,
\qquad B_{n,n}(x_1)=x_1^n.
\]
The smallest case not belonging to these sequences is 
$B_{3,2}(x_1,x_2)=3x_1x_2$.

\medskip
An important application of Bell polynomials, 
which can also be found in \cite{Com}, is
Fa\`a di Bruno's formula on higher derivatives of composite functions:

Let $f$ and $g$ be functions for which all relevant derivatives are defined, and
as usual let $g'(x), g''(x),\ldots, g^{(k)}(x)$ be the first to the 
$k$th derivatives of $g$. Then
\begin{equation}\label{4.10}
\frac{d^n}{dx^n}f(g(x))
=\sum_{k=1}^n f^{(k)}(g(x))\cdot B_{n,k}(g'(x),g''(x),\ldots,g^{(n-k+1)}(x)).
\end{equation}
It is clear from \eqref{4.5} and \eqref{4.6} that the second part of the 
following application of Fa\`a di Bruno's formula is relevant to Lemma~3; the
first part will be used later. 

\begin{lemma}
For any positive integer $n$ we have
\begin{align}
\frac{d^n}{dx^n}\Gamma(y)
&= \Gamma(y)\sum_{k=1}^n B_{n,k}(\psi(y),\psi_1(y),\ldots,\psi_{n-k}(y)),\label{4.11}\\
\frac{d^n}{dx^n}\frac{1}{\Gamma(y)} &= \frac{1}{\Gamma(y)}
\sum_{k=1}^n(-1)^k B_{n,k}(\psi(y),\psi_1(y),\ldots,\psi_{n-k}(y)),\label{4.12}
\end{align}
where $\psi(y)$ is the digamma function defined in \eqref{4.1}, and $\psi_j(y)$,
$j\geq 1$, is the polygamma function of order $j$, i.e., 
$\psi_j(y)=\psi^{(j)}(y)$. We also set $\psi_0(y)=\psi(y)$.
\end{lemma}

\begin{proof}
We write $\Gamma(y)=e^{\log\Gamma(y)}$ and use Fa\`a di Bruno's formula with
$f(x)=e^x$ and $g(y)=\log\Gamma(y)$. Then $f^{(k)}(g(y))=\Gamma(y)$, and
\eqref{4.11} follows immediately from \eqref{4.10}.
The proof of \eqref{4.12} is almost identical, with the difference that
$f(x)=e^{-x}$, and thus $f^{(k)}(g(y))=(-1)^k/\Gamma(y)$.
\end{proof}

We are now ready to state and prove the main result of this section.

\begin{theorem}
Let $\chi$ be a primitive nonprincipal Dirichlet character with conductor $q$.
Then for all integers $n\geq 1$ we have
\begin{equation}\label{4.13}
L_n(\chi) = \frac{1}{q^n}{\sum}^*\prod_{j\in\Phi}\frac{\chi(j)^{k_j}}{k_j!}
\sum_{k=1}^{k_j}(-1)^k B_{k_j,k}\left(\psi(\tfrac{j}{q}),\psi_1(\tfrac{j}{q}),\ldots,\psi_{k_j-k}(\tfrac{j}{q})\right),
\end{equation}
with the index set $\Phi:=\{j\mid 1\leq j\leq q-1, \gcd(j,q)=1\}$, and where
the summation $\sum^*$ is over all $k_j (j\in\Phi)$ that sum to $n$.
\end{theorem}

\begin{proof}
By \eqref{4.5} we have 
\[
f_j^{(k_j)}(0) = \left(\frac{-\chi(j)}{q}\right)^{k_j}
\frac{d^{k_j}}{dy^{k_j}}\left.\frac{1}{\Gamma(y)}\right|_{y=j/q}.
\]
We combine this with \eqref{4.12} and \eqref{3.6}, and note that by \eqref{3.5}
and the fact that $\sum^*k_j=n$, we have
\[
\prod_{j\in\Phi}\frac{(-1)^{k_j}}{q^{k_j}\Gamma(\tfrac{j}{k})}
=\frac{(-1)^n}{q^n}\cdot\frac{\epsilon(q)}{(2\pi)^{\varphi(q)/2}}.
\]
The desired identity \eqref{4.13} now follows immediately from \eqref{4.6}.
\end{proof}

We illustrate this result with two examples.

\medskip
\noindent
{\bf Example~10.} Let $n=1$ in \eqref{4.13}. Then the sum $\sum^*$ is over all
$j\in\Phi$, with $k_j=1$, and the product that follows has the single factor 
for the index $j$ that belongs to $k_j=1$. Since $B_{1,1}(x_1)=x_1$ (see
Example~9), Theorem~3 gives
\[
L_1(\chi) = \frac{1}{q}\sum_{j\in\Phi}\chi(j)(-1)\psi(\tfrac{j}{q}),
\]
which agrees with \eqref{4.2}.

\medskip
\noindent
{\bf Example~11.}  Let $n=2$ in \eqref{4.13}. Then $\sum^*$ consists of two
types of sums:

(i) $k_j=2$ for $j\in\Phi$. This gives the summands
\begin{align*}
\sum_{j\in\Phi}&\frac{\chi(j)^2}{2}
\sum_{j=1}^2(-1)^jB_{2,j}(\psi(\tfrac{j}{q}),\ldots,\psi_{2-k}(\tfrac{j}{q}))\\
&=\sum_{j\in\Phi}\frac{\chi(j)^2}{2}
\left(-B_{2,1}(\psi(\tfrac{j}{q}),\psi_1(\tfrac{j}{q}))\right)
+B_{2,2}(\psi(\tfrac{j}{q}))\\
&=\sum_{j\in\Phi}\frac{\chi(j)^2}{2}
\left(-\psi_1(\tfrac{j}{q})+\psi(\tfrac{j}{q})^2\right),
\end{align*}
where we have used Example~9.

(ii) We sum over $j,\ell\in\Phi$, $j<\ell$; then $k_j=k_{\ell}=1$. For 
simplicity we consider all $1\leq j<\ell\leq q-1$; we can do this since 
$\chi(j)=0$ when $j\not\in\Phi$. Hence we get the summands
\[
\sum_{1\leq j<\ell\leq q-1}\frac{\chi(j)}{1}(-1)B_{1,1}(\psi(\tfrac{j}{q}))
\frac{\chi(\ell)}{1}(-1)B_{1,1}(\psi(\tfrac{\ell}{q}))\\
=\sum_{1\leq j<\ell\leq q-1}\chi(j\ell)\psi(\tfrac{j}{q})\psi(\tfrac{\ell}{q}),
\]
where we have again used Example~9. These two parts, together with \eqref{4.13}
and after some simple manipulation, give 
\begin{equation}\label{4.14}
L_2(\chi)
= \frac{1}{2q^2}\left[\left(\sum_{j=1}^{q-1}\chi(j)\psi(\tfrac{j}{q})\right)^2
-\sum_{j=1}^{q-1}\chi(j)^2\psi_1(\tfrac{j}{q})\right];
\end{equation}
this concludes Example~11.

\medskip
We can somewhat simplify \eqref{4.14} if we note that $\chi(q-j)^2=\chi(j)^2$ 
and that differentiating \eqref{4.3a} gives
\[
\psi_1(z)+\psi_1(1-z) = \frac{\pi^2}{\sin^2(\pi z)}.
\]
Then \eqref{4.14} can be rewritten as follows.

\begin{corollary}
Let $\chi$ be a primitive nonprincipal character with conductor $q>1$. Then
\begin{equation}\label{4.15}
\sum_{1\leq k<\ell}\frac{\chi(k)}{k}\frac{\chi(\ell)}{\ell}
= \frac{1}{2q^2}\left[\left(\sum_{j=1}^{q-1}\chi(j)\psi(\tfrac{j}{q})\right)^2
-\sum_{j=1}^{\lfloor\frac{q-1}{2}\rfloor}
\left(\frac{\chi(j)\pi}{\sin(\tfrac{\pi j}{q})}\right)^2\right],
\end{equation}
and when $\chi$ is odd, then
\begin{equation}\label{4.16}
\sum_{1\leq k<\ell}\frac{\chi(k)}{k}\frac{\chi(\ell)}{\ell}
= \frac{\pi^2}{2q^2}\left[\left(\sum_{j=1}^{\lfloor\frac{q-1}{2}\rfloor}
\chi(j)\cot(\tfrac{\pi j}{q})\right)^2
-\sum_{j=1}^{\lfloor\frac{q-1}{2}\rfloor}
\left(\frac{\chi(j)}{\sin(\tfrac{\pi j}{q})}\right)^2\right].
\end{equation}
\end{corollary}
The last identity follows from \eqref{4.15}, just as \eqref{4.3} follows from 
\eqref{4.2}. Corollary~4 can be seen as the double sum analogue of 
Corollary~2.

To conclude this section, we derive an evaluation for $L_n(\chi)$ for odd 
$\chi$, which is quite different from \eqref{4.13}. For what follows, we set
\[
\Phi_q := \{j\mid 1\leq j\leq\tfrac{q-1}{2},\gcd(j,q)=1\}\quad\hbox{and}\quad
\ell:=|\Phi_q|=\tfrac{1}{2}\varphi(q).
\]

\begin{theorem}
Let $\chi$ be a primitive nonprincipal odd Dirichlet character with conductor 
$q$. Then for all integers $n\geq 1$ we have
\begin{equation}\label{4.17}
L_n(\chi) = \frac{i^{n-\ell}}{\epsilon(q)n!}\bigl(\frac{\pi}{q}\bigr)^n
\sum_{\delta}P(\delta)e^{i\pi j(\delta)/q}
\left(\delta_1\chi(j_1)+\dots+\delta_{\ell}\chi(j_{\ell})\right)^n,
\end{equation}
where the sum is taken over all $2^\ell$ values of 
$\delta=(\delta_1,\ldots,\delta_{\ell})$, $\delta_{\lambda}=\pm 1$, and where
$P(\delta):=\delta_1\dots\delta_{\ell}$, $j_{\lambda}\in\Phi_q$, 
$\lambda=1,\ldots,\ell$, and 
$j(\delta):=\delta_1j_1+\dots+\delta_{\ell}j_{\ell}$.
\end{theorem}

\begin{proof}
Our point of departure is the identity \eqref{2.4}, slightly rewritten in the
form
\begin{equation}\label{4.18}
\prod_{k=1}^\infty\left(1-\chi(k)\frac{z}{k}\right)
=\frac{2^\ell}{\epsilon(q)}\cdot\prod_{\lambda=1}^{\ell}
\sin\left(\frac{\pi}{q}j_\lambda-\chi(j_\lambda)\frac{\pi z}{q}\right),
\end{equation}
with $j_\lambda$ as in the statement of the theorem. Writing the sine function
in terms of the exponential function, the right-hand side of \eqref{4.18}
becomes
\begin{align}
\frac{1}{\epsilon(q)i^\ell}&\prod_{\lambda=1}^{\ell}
\left[\exp\left(\frac{i\pi}{q}j_\lambda-\chi(j_\lambda)\frac{i\pi z}{q}\right)
-\exp\left(-\frac{i\pi}{q}j_\lambda+\chi(j_\lambda)\frac{i\pi z}{q}\right)\right] \label{4.19}\\
&=\frac{1}{\epsilon(q)i^\ell}\sum_{\delta}P(\delta)\prod_{\lambda=1}^{\ell}
\exp\left(\frac{i\pi}{q}\delta_\lambda j_\lambda
-\delta_\lambda\chi(j_\lambda)\frac{i\pi z}{q}\right)\nonumber\\
&=\frac{1}{\epsilon(q)i^\ell}\sum_{\delta}P(\delta)
\left(\prod_{\lambda=1}^{\ell}\exp\left(\frac{i\pi}{q}\delta_\lambda j_\lambda\right)\right)
\left(\prod_{\lambda=1}^{\ell}\exp\left(-\delta_\lambda\chi(j_\lambda)\frac{i\pi z}{q}\right)\right)\nonumber
\end{align}
Now by the definition of $j(\delta)$ we have
\begin{equation}\label{4.20}
\prod_{\lambda=1}^{\ell}\exp\left(\frac{i\pi}{q}\delta_\lambda j_\lambda\right)
=\exp\left(\frac{i\pi j(\delta)}{q}\right),
\end{equation}
and the Maclaurin expansion of the exponential function gives
\begin{align*}
\prod_{\lambda=1}^{\ell}\exp\left(-\delta_\lambda\chi(j_\lambda)\frac{i\pi z}{q}\right)
&=\exp\left(-\frac{i\pi z}{q}
\sum_{\lambda=1}^{\ell}\delta_\lambda\chi(j_\lambda)\right)\\
&=\sum_{\nu=0}^\infty\left(-\frac{i\pi}{q}\right)^\nu
\left(\delta_1\chi(j_1)+\dots+\delta_{\ell}\chi(j_{\ell})\right)^\nu
\frac{z^\nu}{\nu!}.
\end{align*}
We now combine this last identity and \eqref{4.20} with \eqref{4.19} and equate
coefficients of $z^n$ with those on the left of \eqref{4.18}, which are 
$(-1)^nL_n(\chi)$. This gives
\[
(-1)^nL_n(\chi) = \frac{1}{\epsilon(q)i^\ell n!}\left(-\frac{i\pi}{q}\right)^n
\sum_{\delta}P(\delta)e^{i\pi j(\delta)/q}
\left(\delta_1\chi(j_1)+\dots+\delta_{\ell}\chi(j_{\ell})\right)^n.
\]
This is clearly the same as \eqref{4.17}, and the proof is complete.
\end{proof}

\noindent
{\bf Example~12.} Let $q=3$. Then $\ell=1$, $j_1=1$, $\delta=(\delta_1)$ with
$\delta_1=\pm 1$, $j(\delta)=\delta_1j_1=\delta_1$, $P(\delta)=\delta_1$, and
$\epsilon(q)=\sqrt{3}$. Then \eqref{4.17} reduces to
\[
L_n(\chi) = \frac{i^{n-1}}{n!\sqrt{3}}\left(e^{i\pi/3}-(-1)^ne^{-i\pi/3}\right).
\]
This identity combines the two parts of \eqref{4.4b}. This is clear if we 
consider even and odd $n$ separately and use $\sin(\pi/3)=\sqrt{3}/2$,
$\cos(\pi/3)=1/2$. The cases $q=4$ and $q=6$ can be obtained in a similar way.

\section{Multiple $L$-star series}

In the theory of multiple zeta functions and multiple $L$-series one usually
distinguishes between the regular case where, as in \eqref{4.3b} and 
\eqref{4.4}, each summation index is strictly greater than the previous one,
and the ``star analogue", where the inequalities are not strict. Accordingly,
in analogy to \eqref{4.4} we define
\begin{equation}\label{4a.1}
L_n^*(\chi)
:=\sum_{1\leq k_1\leq\dots\leq k_n}\frac{\chi(k_1)}{k_1}\dots\frac{\chi(k_n)}{k_n},
\end{equation}
where, as usual, $\chi$ is a primitive nonprincipal Dirichlet character with
conductor $q$. While the generating function for $L_n(\chi)$, say $G(z)$, is 
given by the infinite product of Theorems~1 and~2 (see \eqref{4.4a}), we will
now see that the generating function for $L_n^*(\chi)$, say $H(z)$, is related
to $G(z)$ by the simple identity
\[
H(z) = \frac{1}{G(-z)}.
\]
Similar relations were previously observed and used by other authors
in different settings; see, e.g., \cite[Sect.~5]{CC} or \cite{Ho2}.

\begin{lemma}
Let $\chi$ be a primitive nonprincipal Dirichlet character. Then
\begin{equation}\label{4a.2}
\prod_{k=1}^\infty\left(1-\chi(k)\frac{z}{k}\right)^{-1}
= 1 + \sum_{n=1}^\infty L_n^*(\chi)z^n.
\end{equation}
\end{lemma}

\begin{proof}
Expanding the factors on the left-hand side of \eqref{4a.2}, we get
\[
\prod_{k=1}^\infty\left(1+\tfrac{\chi(k)}{k}z+(\tfrac{\chi(k)}{k})^2z^2
+(\tfrac{\chi(k)}{k})^3z^3+\dots\right).
\]
If we fix an $n\geq 1$ and consider the coefficients of $z^n$, we see that
their sum is the right-hand side of \eqref{4a.1}.
\end{proof}

As an immediate consequence, obtained by taking the product of \eqref{4a.2} and
\eqref{4.4a} and equating coefficients of $z^n$, we get the following 
connections between the series $L_n(\chi)$ and $L_n^*(\chi)$.

\begin{corollary}
Let $\chi$ be a primitive nonprincipal Dirichlet character. Then we have for
all $n\geq 1$,
\begin{equation}\label{4a.3}
L_n^*(\chi)+\sum_{j=1}^{n-1}(-1)^jL_j(\chi)L_{n-j}^*(\chi)+(-1)^nL_n(\chi)=0.
\end{equation}
\end{corollary}

For $n=1$, this reduces to the obvious fact that $L^*_1(\chi)=L_1(\chi)$
as single series.

Using \eqref{4a.2} and the identities \eqref{2.2}, \eqref{2.7a} and \eqref{2.7b},
we now derive $L^*$-analogues of the identities in Corollary~3. This case is
more involved and requires Bernoulli and Euler numbers and polynomials.
The Bernoulli polynomials $B_n(z)$ can be defined by the generating function
\begin{equation}\label{4a.4}
\frac{xe^{zx}}{e^x-1} = \sum_{n=0}^\infty B_n(z)\frac{x^n}{n!}\qquad(|x|<2\pi),
\end{equation}
and the Bernoulli numbers are then $B_n:=B_n(0)$, $n=0, 1, 2,\dots$. The Euler
polynomials $E_n(z)$ can be defined by
\begin{equation}\label{4a.5}
\frac{2e^{zx}}{e^x+1} = \sum_{n=0}^\infty E_n(z)\frac{x^n}{n!}\qquad(|x|<\pi),
\end{equation}
and the Euler numbers are then defined by
\begin{equation}\label{4a.6}
E_n := 2^nE_n(\tfrac{1}{2}),\qquad n=0,1,2,\ldots.
\end{equation}
Numerous properties of these numbers and polynomials can be found, e.g., in
\cite[Ch.~24]{DLMF}, and some special values are listed in Table~1 below.

\begin{corollary}
$(a)$ If $\chi$ is the nonprincipal character with $q=3$, then
\begin{align}
L_{2n}^*(\chi) &= (-1)^{n+1}3(2^{2n}+1)\frac{B_{2n+1}(\tfrac{1}{3})}{(2n+1)!}\pi^{2n}\qquad(n\geq 1),\label{4a.7}\\
L_{2n+1}^*(\chi) &=
(-1)^n\frac{\sqrt{3}}{2}(2^{2n+1}-1)(3^{2n+2}-1)\frac{B_{2n+2}}{(2n+2)!}\left(\frac{\pi}{3}\right)^{2n+1}\quad(n\geq 0)\label{4a.8}.
\end{align}
$(b)$ If $\chi$ is the nonprincipal character with $q=4$, then
\begin{equation}\label{4a.9}
L_n^*(\chi) = (-1)^{\lfloor\frac{n+1}{2}\rfloor}\frac{E_n(\tfrac{1}{4})}{n!}\pi^n\qquad(n\geq 1).
\end{equation}
$(c)$ If $\chi$ is the nonprincipal character with $q=6$, then
\begin{align}
L_{2n}^*(\chi) &= (-1)^n\frac{1}{4}(3^{2n+1}+1)
\frac{E_{2n}}{(2n)!}\left(\frac{\pi}{6}\right)^{2n}\qquad(n\geq 1),\label{4a.10}\\
L_{2n+1}^*(\chi) &=
(-1)^{n+1}\frac{\sqrt{3}}{2}\frac{E_{2n+1}(\tfrac{1}{6})}{(2n+1)!}\pi^{2n+1}\qquad(n\geq 0).\label{4a.11}
\end{align}
\end{corollary}

Before proving these identities, we list the special numbers that occur in the
identities \eqref{4a.7}--\eqref{4a.11}.

\bigskip
\begin{center}
{\renewcommand{\arraystretch}{1.3}
\begin{tabular}{|r||r|r|r|r|r|}
\hline
$n$ & $B_{2n+1}(\tfrac{1}{3})$ & $B_{2n+2}$ & $E_n(\tfrac{1}{4})$ & $E_{2n}$ & $E_{2n+1}(\tfrac{1}{6})$\\
\hline
0 & $-\frac{1}{6}$ & $\frac{1}{6}$ & 1 & 1 & $-\frac{1}{3}$ \\
1 & $\frac{1}{27}$ & $-\frac{1}{30}$ & $-\frac{1}{4}$ & $-1$ & $\frac{23}{108}$ \\
2 & $-\frac{5}{243}$ & $\frac{1}{42}$ & $-\frac{3}{16}$ & 5 & $-\frac{1681}{3888}$\\
3 & $\frac{49}{2187}$ & $-\frac{1}{30}$ & $\frac{11}{64}$ & $-61$ & $\frac{257543}{139968}$\\
4 & $-\frac{809}{19683}$ & $\frac{5}{66}$ & $\frac{57}{256}$ & 1385 & $-\frac{67637281}{5038848}$\\
5 & $\frac{20317}{177147}$ & $-\frac{691}{2730}$ & $-\frac{361}{1024}$ & $-50521$ & $\frac{27138236663}{181398528}$\\
\hline
\end{tabular}}

\medskip {\bf Table~1}: $B_{2n+1}(\tfrac{1}{3}), B_{2n+2}, E_n(\tfrac{1}{4}), E_{2n}, E_{2n+1}(\tfrac{1}{6})$ for $0\leq n\leq 5$.  \end{center}

\medskip
\begin{proof}[Proof of Corollary~6]
(a) By \eqref{4a.2} and \eqref{2.7a}, $L_n^*(\chi)$ is the $n$th Taylor 
coefficient in the expansion of $\frac{1}{2}\sqrt{3}\csc(\frac{\pi}{3}(1-z))$
about $z=0$. For ease of notation we set $y:=\pi z/3$. Using the addition 
formula for the sine function, we obtain
\begin{align*}
\csc(\tfrac{\pi}{3}(1-z)) &= \frac{2}{\sqrt{3}\cos{y}-\sin{y}}
=2\cdot\frac{\sqrt{3}\cos{y}+\sin{y}}{3\cos^2y-\sin^2y}\\
&= \frac{2\sqrt{3}\cos{y}}{4\cos^2y-1}+\frac{2\sin{y}}{3-4\sin^2y} \\
&= \left(\frac{\sqrt{3}}{1+2\cos{y}}+\frac{\sqrt{3}}{4\cos^2y-1}\right)
+\left(\frac{3/2}{-4\sin^3y+3\sin{y}}-\frac{1/2}{\sin{y}}\right),
\end{align*}
where we have used a partial fraction decomposition in each of the two summands
in the second row. Next, using the identities 
\[
\cos(2y)=2\cos^2y-1,\qquad \sin(3y)=-4\sin^3y+3\sin{y}.
\]
we get
\[ \csc(\tfrac{\pi}{3}(1-z)) = 
\left(\frac{\sqrt{3}}{1+2\cos{y}}+\frac{\sqrt{3}}{1+2\cos(2y)}\right)
+\left(\frac{3/2}{\sin(3y)}-\frac{1/2}{\sin{y}}\right),
\]
and thus
\begin{align}
\frac{\sqrt{3}}{2}\csc(\tfrac{\pi}{3}(1-z)) 
&= \frac{3/2}{1+e^{\pi iz/3}+e^{-\pi iz/3}}
+\frac{3/2}{1+e^{2\pi iz/3}+e^{-2\pi iz/3}} \label{4a.12}\\
&\qquad+\frac{\sqrt{3}}{4}\left(3\csc(\pi z)-\csc(\tfrac{1}{3}\pi z)\right).\nonumber
\end{align}
Now we use the identity
\begin{equation}\label{4a.13}
\frac{3/2}{1+e^x+e^{-x}}
=-\sum_{n=0}^\infty\frac{3^{2n+1}}{2n+1}B_{2n+1}(\tfrac{1}{3})\frac{x^{2n}}{(2n)!},
\end{equation}
which is due to Glaisher \cite[p.~35]{Gl} (but note his different normalization
of Bernoulli polynomials), and the well-known expansion
\begin{equation}\label{4a.14}
\csc{x} = \sum_{n=0}^\infty(-1)^{n-1}\frac{2(2^{2n-1}-1)}{(2n)!}B_{2n}x^{2n-1};
\end{equation}
see, e.g., \cite[Eq.~4.19.4]{DLMF}. The identity \eqref{4a.13}, applied to the
first two terms on the right of \eqref{4a.12}, leads to \eqref{4a.7}, and 
\eqref{4a.14} applied to the final term of \eqref{4a.12} gives \eqref{4a.8}.

(b) By \eqref{4a.2} and \eqref{2.2}, $L_n^*(\chi)$ is the $n$th Taylor
coefficient in the expansion of $\frac{1}{2}\sqrt{2}\csc(\frac{\pi}{4}(1-z))$
about $z=0$. Using the addition formula for sine again, we see that
\begin{equation}\label{4a.15}
\frac{1}{2}\sqrt{2}\csc(\tfrac{\pi}{4}(1-z)) 
= \frac{1}{\cos(\pi z/4)-\sin(\pi z/4)}.
\end{equation}
We now use the known expansion
\begin{equation}\label{4a.16}
\frac{1}{\cos{x}-\sin{x}} = \sum_{n=0}^\infty Q_n(1)\frac{x^n}{n!},
\end{equation}
where the numbers $Q_n(1)$ are known as Springer numbers; see \cite{Ho1}. They
were also studied by Glaisher, and Hoffman \cite[Eq.~(12)]{Ho1} showed that
\begin{equation}\label{4a.17}
Q_n(1) = (-1)^{\lfloor(n+1)/2\rfloor}4^nE_n(\tfrac{1}{4})\qquad(n=0,1,2,\ldots).
\end{equation}
Combining \eqref{4a.15}--\eqref{4a.17}, we immediately get \eqref{4a.9}.

(c) The proof of \eqref{4a.10} and \eqref{4a.11} is similar in nature to that
of part (a). Here we begin with \eqref{4a.2} and \eqref{2.7b}; we leave the
details to the reader.
\end{proof}

It should be mentioned that all 5 identities in Corollary~6 were first found
experimentally by using MAPLE and the OEIS \cite{OEIS}.
We now prove the $L$-star analogue of the general Theorem~3.

\begin{theorem}
Let $\chi$ be a primitive nonprincipal Dirichlet character with conductor $q$.
Then for all integers $n\geq 1$ we have
\begin{equation}\label{4a.18}
L_n^*(\chi) 
= \frac{(-1)^n}{q^n}{\sum}^*\prod_{j\in\Phi}\frac{\chi(j)^{k_j}}{k_j!}
\sum_{k=1}^{k_j}B_{k_j,k}\left(\psi(\tfrac{j}{q}),\psi_1(\tfrac{j}{q}),\ldots,\psi_{k_j-k}(\tfrac{j}{q})\right),
\end{equation}
with the index set $\Phi:=\{j\mid 1\leq j\leq q-1, \gcd(j,q)=1\}$, and where
the summation $\sum^*$ is over all $k_j (j\in\Phi)$ that sum to $n$.
\end{theorem}

\begin{proof}
We proceed as in the proofs of Lemma~3 and Theorem~3. If we combine \eqref{4a.2}
with \eqref{2.3} and use \eqref{4.8} with $f_j(z)$ replaced by 
\begin{equation}\label{4a.19}
g_j(z) := \Gamma\left(\frac{j-\chi(j)z}{q}\right),\qquad
1\leq j\leq q-1,\; \gcd(j,q)=1,
\end{equation}
then we get the following analogue of \eqref{4.6}:
\begin{equation}\label{4a.20}
L_n^*(\chi) = \frac{\epsilon(q)}{(2\pi)^{\varphi(q)/2}}
{\sum}^*\prod_{j\in\Phi}\frac{1}{k_j!}g_j^{(k_j)}(0),
\end{equation}
where the sum and the product are over the same ranges as before. By
\eqref{4a.19} and \eqref{4.11} we have
\begin{align}
g_j^{(k_j)}(0) &= \left(\frac{-\chi(j)}{q}\right)^{k_j}
\left.\frac{d^{k_j}}{dy^{k_j}}\Gamma(y)\right|_{y=j/q}\label{4a.21}\\
&= \left(\frac{-\chi(j)}{q}\right)^{k_j}\Gamma(\tfrac{j}{k})
\sum_{k=1}^{k_j}B_{k_j,k}\left(\psi(\tfrac{j}{q}),\psi_1(\tfrac{j}{q}),\ldots,\psi_{k_j-k}(\tfrac{j}{q})\right).\nonumber
\end{align}
As before, since $\sum^*k_j=n$, and using \eqref{3.5}, we have
\[
\prod_{j\in\Phi}\left(\frac{-1}{q}\right)^{k_j}\Gamma(\tfrac{j}{k})
=\left(\frac{-1}{q}\right)^n\frac{(2\pi)^{\varphi(q)/2}}{\epsilon(q)}.
\]
Combining this with \eqref{4a.21} and \eqref{4a.20}, we finally get
\eqref{4a.18}.
\end{proof}

We illustrate Theorem~5 with the following example, which is analogous to
Example~11.

\medskip
\noindent
{\bf Example~13.}  Let $n=2$ in \eqref{4a.18}. Then, just as in Example~11,
$\sum^*$ consists of the two types of sums
\[
\sum_{j\in\Phi}\frac{\chi(j)^2}{2}
\left(B_{2,1}(\psi(\tfrac{j}{q}),\psi_1(\tfrac{j}{q}))
+B_{2,2}(\psi(\tfrac{j}{q}))\right)
=\sum_{j\in\Phi}\frac{\chi(j)^2}{2}
\left(\psi_1(\tfrac{j}{q})+\psi(\tfrac{j}{q})^2\right)
\]
and
\[
\sum_{1\leq j<\ell\leq q-1}\frac{\chi(j)}{1}B_{1,1}(\psi(\tfrac{j}{q}))
\frac{\chi(\ell)}{1})B_{1,1}(\psi(\tfrac{\ell}{q}))\\
=\sum_{1\leq j<\ell\leq q-1}\chi(j\ell)\psi(\tfrac{j}{q})\psi(\tfrac{\ell}{q}),
\]
where, as before, we have used Example~9. The two sums, together with 
\eqref{4a.18}, give
\begin{equation}\label{4a.22}
L_2^*(\chi)
= \frac{1}{2q^2}\left[\left(\sum_{j=1}^{q-1}\chi(j)\psi(\tfrac{j}{q})\right)^2
+\sum_{j=1}^{q-1}\chi(j)^2\psi_1(\tfrac{j}{q})\right].
\end{equation}

\medskip
Note the difference in sign from \eqref{4.14}. The connection between these
two identities also follows from Corollary~5 with $n=2$, which gives
\[
L_2^*(\chi)+L_2(\chi) = L_1(\chi)L_1^*(\chi) = L_1(\chi)^2;
\]
then compare this with \eqref{4.2}. The identity \eqref{4a.22} 
immediately leads to analogues of \eqref{4.15} and \eqref{4.16}. Finally, if we
subtract \eqref{4.15}, resp.\ \eqref{4.16}, from these analogues, we get the 
following evaluation in both cases.

\begin{corollary}
Let $\chi$ be a primitive nonprincipal character with conductor $q>1$. Then
\begin{equation}\label{4a.23}
\sum_{k=1}^\infty\frac{\chi(k)^2}{k^2}
= \frac{\pi^2}{q^2}\sum_{j=1}^{\lfloor\frac{q-1}{2}\rfloor}
\left(\frac{\chi(j)}{\sin(\tfrac{\pi j}{q})}\right)^2.
\end{equation}
In particular,
\begin{equation}\label{4a.24}
\sum_{\substack{k=1\\(k,q)=1}}^\infty\frac{1}{k^2}
=\frac{\pi^2}{q^2}\sum_{{\substack{j=1\\(j,q)=1}}}^{\lfloor\frac{q-1}{2}\rfloor}
\csc^2(\tfrac{\pi j}{q}).
\end{equation}
\end{corollary}

The identity \eqref{4a.24} follows from \eqref{4a.23} by letting $\chi$ be the
unique quadratic character modulo $q$.

\medskip
\noindent
{\bf Example~14.}  Let $q$ be an odd prime in \eqref{4a.24}. Then the left-hand
side becomes $\zeta(2)-\zeta(2)/q^2$, and using Euler's evaluation 
$\zeta(2)=\pi^2/6$, we get
\[
\sum_{j=1}^{\frac{q-1}{2}}\csc^2(\tfrac{\pi j}{q}) = \frac{q^2-1}{6}.
\]
This is a special case of a known identity; see, e.g., \cite[p.~644]{PrE} or
\cite[(24.1.3)]{Ha}.

\section{Products involving cyclotomic polynomials}

The first identity in \eqref{1.2} is a special case of 
\begin{equation}\label{5.1}
\prod_{k=1}^\infty\left(1-(\tfrac{z}{k})^m\right)
=\frac{-1}{z^m}\prod_{j=1}^m\Gamma(-ze^{2\pi ij/m})^{-1}\qquad(m\geq 2),
\end{equation}
which can be found, with a sketch of a proof, as identity 1.3(7) in 
\cite[p.~7]{EMOT}. It is the purpose of this section to derive identities 
similar to \eqref{5.1}, but involving products of cyclotomic polynomials on
the left. We begin with a general result.

\begin{lemma}
Suppose that the complex-valued function $f$ is such that $f(1)+\dots+f(m)=0$,
and set $P_f:=f(1)\dots f(m)$. If $P_f\neq 0$, then
\begin{equation}\label{5.2}
\prod_{k=1}^\infty\prod_{j=1}^m\left(1-f(j)\tfrac{z}{k}\right)
=\frac{(-1)^m}{z^mP_f}\prod_{j=1}^m\frac{1}{\Gamma(-zf(j))}
= \prod_{j=1}^m\frac{1}{\Gamma(1-zf(j))}.
\end{equation}
\end{lemma}

\begin{proof}
Using the second identity in \eqref{1.3}, we get
\begin{align*}
\prod_{j=1}^m\frac{1}{\Gamma(-zf(j))}
&=\prod_{j=1}^m\left\{-zf(j)e^{\gamma(-zf(j))}\prod_{k=1}^\infty
\left[\left(1-\frac{zf(j)}{k}\right)e^{zf(j)/k}\right]\right\} \\
&=(-z)^mP_f\exp(-z\gamma\textstyle{\sum_{j=1}^m f(j)})\\
&\quad\cdot\prod_{k=1}^\infty\left[\prod_{j=1}^m\left(1-\frac{zf(j)}{k}\right)
\exp(\tfrac{z}{k}\textstyle{\sum_{j=1}^m f(j)})\right].
\end{align*}
Now the condition $f(1)+\dots+f(m)=0$ immediately gives the first identity of
\eqref{5.2}. The second identity follows from the basic equation
$x\Gamma(x)=\Gamma(x+1)$.
\end{proof}

\medskip
\noindent
{\bf Example~15.} Let $f(j)=e^{2\pi ij/m}$. Then for any positive integer $k$ 
we have the obvious factorization
\[
\prod_{j=1}^m\left(f(j)\tfrac{z}{k}-1\right) = (\tfrac{z}{k})^m-1,
\]
which illustrates the well-known fact that the sum and the product of 
$f(1),\ldots,f(m)$ are 0 and $-1$, respectively. The identity \eqref{5.1} now
follows immediately from \eqref{5.2}.

\medskip
If instead of taking all the $m$th roots of unity, as done in Example~15, we 
consider only the primitive $m$th roots of unity, we obtain a result concerning
cyclotomic polynomials. In what follows, let $\Phi_m(x)$ be the $m$th 
cyclotomic polynomial.

\begin{theorem}
If $m\geq 3$ is an integer that contains a square, then
\begin{equation}\label{5.3}
\prod_{k=1}^\infty\Phi_m(\tfrac{z}{k})
=\frac{1}{z^{\varphi(m)}}\prod_{\substack{j=1\\(j,m)=1}}^{m-1}
\frac{1}{\Gamma(-ze^{2\pi ij/m})}
= \prod_{\substack{j=1\\(j,m)=1}}^{m-1}\frac{1}{\Gamma(1-ze^{2\pi ij/m})}.
\end{equation}
\end{theorem}

\begin{proof}
We let $f(j)=e^{2\pi ij/m}$ when $\gcd(j,m)=1$, and take the product in
\eqref{5.2} over only those $j$. Then by definition of the cyclotomic polynomial
we have
\[
\prod_{\substack{j=1\\(j,m)=1}}^m\left(1-f(j)\tfrac{z}{k}\right)
= (-1)^{\varphi(m)}\Phi_m(\tfrac{z}{k}).
\]
We recall the well-known identity
\begin{equation}\label{5.4}
\Phi_m(x) = \prod_{d\mid m}\left(x^d-1\right)^{\mu(m/d)},
\end{equation}
where $\mu(n)$ is the M\"obius function. From the properties of $\mu(n)$ and
\eqref{5.4} we know that the product of all primitive $m$th roots of unity is 1,
i.e., $P_f=1$, and their sum is $\mu(m)$. Hence the summation condition is
satisfied when $\mu(m)=0$, i.e., when $m$ is not square-free. Now \eqref{5.3}
follows directly from \eqref{5.2}.
\end{proof}

Using the reflection formula \eqref{3.7} for the gamma function, we can rewrite
the identity \eqref{5.3} for even integers $m$. We first require the following
definition. For integers $m\geq 2$ we set
\begin{equation}\label{5.5}
S_2(m):= \sum_{\substack{j=1\\(j,m)=1}}^{\lfloor m/2\rfloor}j.
\end{equation}
We will deal with the evaluation of this sum later in this section and in
Section~6.

\begin{corollary} 
If $m\geq 4$ is an even integer that contains a square, then
\begin{equation}\label{5.6}
\prod_{k=1}^\infty\Phi_m(\tfrac{z}{k})
=\frac{e^{2\pi iS_2(m)/m}}{(-\pi z)^{\varphi(m)/2}}
\prod_{\substack{j=1\\(j,m)=1}}^{\tfrac{m}{2}-1}\sin(\pi ze^{2\pi ij/m}).
\end{equation}
\end{corollary}

Before proving this identity, we consider two examples.

\medskip
\noindent
{\bf Example~16.} Let $m=4$. Then $\Phi_4(\frac{z}{k})=1+(\frac{z}{k})^2$,
$S_2(4)=1$, $\varphi(4)=2$, and $e^{2\pi i/4}=1$. Hence we get from \eqref{5.6},
\begin{equation}\label{5.7}
\prod_{k=1}^\infty\left(1+(\tfrac{z}{k})^2\right)
=\frac{i}{-\pi z}\sin(\pi iz) = \frac{\sinh(\pi z)}{\pi z}.
\end{equation}
This is a well-known identity; see, e.g., \cite[(89.5.16)]{Ha}. With $z=1$, 
this gives
\begin{equation}\label{5.8}
\prod_{k=1}^\infty\left(1+\frac{1}{k^2}\right)
=\frac{1}{2\pi}\left(e^{\pi}-e^{-\pi}\right).
\end{equation}
It is interesting to note that Theorem~6 gives the evaluation 
$(\Gamma(i)\Gamma(-i))^{-1}$ for \eqref{5.8}.
Finally, if we divide both sides of \eqref{5.7} by $1+z^2$ and take the
limit as $z\rightarrow i$, then we get the first identity in \eqref{1.2}.

\medskip
\noindent
{\bf Example~17.} Let $m=12$. Then $\Phi_{12}(x)=1-x^2+x^4$, $S_2(12)=6$, 
$\varphi(12)=4$, and $e^{2\pi i6/12}=-1$. Hence 
\begin{equation}\label{5.9}
\prod_{k=1}^\infty\left(1-(\tfrac{z}{k})^2+(\tfrac{z}{k})^4\right)
=\frac{-1}{\pi^2z^2}\sin(\pi ze^{\pi i/6})\sin(\pi ze^{5\pi i/6}).
\end{equation}
Using the identities 
\[
e^{\pi i/6}=\frac{\sqrt{3}+i}{2},\quad e^{5\pi i/6}=\frac{-\sqrt{3}+i}{2},\quad
\cos(i\theta)=\cosh(\theta),\quad \sin(i\theta)=i\sinh(\theta),
\]
some easy manipulations of the right-hand side of \eqref{5.9} give
\begin{equation}\label{5.10}
\prod_{k=1}^\infty\left(1-(\tfrac{z}{k})^2+(\tfrac{z}{k})^4\right)
=\frac{\sin^2(\frac{1}{2}\sqrt{3}\pi z)+\sinh^2(\frac{1}{2}\pi z)}{\pi^2z^2}.
\end{equation}
If we choose the specific value $z=1/(2\sqrt{3})$, then \eqref{5.10} shows 
after another easy manipulation that
\[
\prod_{k=1}^\infty\left(1-\frac{1}{12k^2}+\frac{1}{144k^4}\right)
=\frac{6}{\pi^2}\cdot\cosh\left(\frac{\pi}{2\sqrt{3}}\right).
\]

\begin{proof}[Proof of Corollary~8]
If we set $z=-y$ in \eqref{3.7} and use the basic identity 
$\Gamma(y+1)=y\Gamma(y)$, then we get the reflection formula
\begin{equation}\label{5.11}
\Gamma(y)\Gamma(-y) = \frac{-\pi}{y\sin(\pi y)}.
\end{equation}
Next we note that for even $m\geq 4$ we have 
\[
-e^{2\pi ij/m} = e^{2\pi i(j/m+1/2)} = e^{2\pi i(\frac{j+m/2}{2})}.
\]
On the right-hand side of \eqref{5.3} we can therefore group the terms belonging
to $j$ and $j+m/2$ together and use \eqref{5.11} with $y=-ze^{2\pi ij/m}$. Then
we get
\begin{equation}\label{5.12}
\prod_{k=1}^\infty\Phi_m(\tfrac{z}{k})
=\frac{1}{z^{\varphi(m)}}\cdot\frac{(-z)^{\varphi(m)/2}}{(-\pi)^{\varphi(m)/2}}
\prod_{\substack{j=1\\(j,m)=1}}^{\tfrac{m}{2}-1}e^{2\pi ij/m}
\prod_{\substack{j=1\\(j,m)=1}}^{\tfrac{m}{2}-1}\sin(-\pi ze^{2\pi ij/m}).
\end{equation}
But this is the same as \eqref{5.6} if we note that by \eqref{5.5} the first 
product on the right of \eqref{5.12} is $e^{2\pi iS_2(m)/m}$.
\end{proof}

While the sum $S_2(m)$, defined in \eqref{5.5}, can be evaluated for all $m$
(see the final section), only when $m$ is divisible by 4 do we get values that
allow for a simplification of \eqref{5.6}. We begin with a lemma.

\begin{lemma}
Let $m\geq 4$ be an integer divisible by 4. Then
\begin{equation}\label{5.13}
\frac{4\,S_2(m)}{m}\equiv\begin{cases}
1\pmod{4} &\hbox{if}\;m=4;\\
2\pmod{4} &\hbox{if}\;m=8\;\hbox{or}\;m=4p^\alpha, p\equiv 3\pmod{4}, \alpha\geq 1;\\
0\pmod{4} &\hbox{otherwise}.
\end{cases}
\end{equation}
\end{lemma}

\begin{proof}
For $m$ divisible by 4, it is known that $S_2(m)=m\varphi(m)/8$;
see \cite[p.~65]{Zm} or \cite[A066840]{OEIS}. Hence
\begin{equation}\label{5.14}
\frac{4\,S_2(m)}{m} = \frac{1}{2}\varphi(m).
\end{equation}
When $m=4$, then the right-hand side of \eqref{5.14} is 1, and for $m=8$ it is
2. For all higher powers of 2, \eqref{5.14} is $2^\nu$ for $\nu\geq 2$.

When $m=4p^\alpha$, with $p$ an odd prime and $\alpha\geq 1$, then
\[
\frac{1}{2}\varphi(m) = (p-1)p^{\alpha-1} \equiv\begin{cases}
0\pmod{4} &\hbox{if}\;\;p\equiv 1\pmod{4},\\
2\pmod{4} &\hbox{if}\;\;p\equiv 3\pmod{4}.
\end{cases}
\]
Finally, when $m=4m'$ and $m'$ has at least two distinct prime factors, then
each of them contributes at least one factor of 2. Hence 
$\frac{1}{2}\varphi(m)\equiv 0\pmod{4}$ in this case, which completes the proof.
\end{proof}

We are now ready to state a simplified special case of Corollary~8.

\begin{corollary}
If $m$ is a positive integer divisible by 4, then
\begin{equation}\label{5.15}
\prod_{k=1}^\infty\Phi_m(\tfrac{z}{k})
=\frac{\delta(m)}{(-\pi z)^{\varphi(m)/2}}
\prod_{\substack{j=1\\(j,m)=1}}^{\tfrac{m}{2}-1}\sin(\pi ze^{2\pi ij/m}),
\end{equation}
where
\[
\delta(m) = \begin{cases}
i &\hbox{if}\;m=4;\\
-1 &\hbox{if}\;m=8\;\hbox{or}\;m=4p^\alpha,\; p\equiv 3\pmod{4},\; \alpha\geq 1;\\
1 &\hbox{otherwise}.
\end{cases}
\]
\end{corollary}

This follows immediately from Corollary~5 and Lemma~5 if we note that
\[
e^{2\pi iS_2(m)/m} = i^{4S_2(m)/m}.
\]
The identities \eqref{5.7} and \eqref{5.9} are examples of the first, resp.\ the
second case of the definition of $\delta(m)$.

To conclude this section, we remark that infinite products involving cyclotomic
polynomials, but from a different point of view and of a different type, were
studied in \cite{DN}.

\section{Further Remarks}

The sum $S_2(m)$, defined in \eqref{5.5}, is an interesting object in its own 
right; we will therefore give closed form expressions in all cases.
In the process we require an auxiliary result concerning the ``full" analogue
\begin{equation}\label{6.1}
S_1(m):= \sum_{\substack{j=1\\(j,m)=1}}^{m-1}j.
\end{equation}

\begin{proposition}
Let $m\geq 2$ be an integer. Then
\begin{equation}\label{6.2}
S_1(m) = \tfrac{1}{2} m\varphi(m).
\end{equation}
If $m$ is odd, with distinct prime factors $p_1,\ldots,p_r$, then
\begin{equation}\label{6.3}
S_2(m) = \tfrac{1}{8}\left(m\varphi(m)-(1-p_1)\dots(1-p_r)\right)
\end{equation}
If $m$ is divisible by $4$, then
\begin{equation}\label{6.4}
S_2(m) = \tfrac{1}{8} m\varphi(m).
\end{equation}
Finally, if $m=2m'$, where $m'$ is odd, then
\begin{equation}\label{6.5}
S_2(m) = \tfrac{1}{4} m\varphi(m) - 2S_2(m').
\end{equation}
\end{proposition}

\begin{proof}
The identity \eqref{6.2} appears in \cite[pp.~42--43]{Na} as Exercise~19, and
\eqref{6.3} as Exercise~36 in \cite[p.~45]{Na}. The identity \eqref{6.4}, with
references, was already mentioned in the proof of Lemma~5 above. It remains to
prove \eqref{6.5}. To do so, we use the definitions \eqref{5.5} and \eqref{6.1}
to get
\begin{align*}
S_2(2m') &= \sum_{\substack{j=1\\(j,2m')=1}}^{m'}j
= \sum_{\substack{j=1\\(j,m')=1}}^{m'}j
-\sum_{\substack{j=1\\(j,m')=1\\j\, {\rm even}}}^{m'}j \\
&= S_1(m')-\sum_{\substack{j'=1\\(j',m')=1}}^{\lfloor m'/2\rfloor}2j'
= S_1(m')-2S_2(m').
\end{align*}
Finally, by \eqref{6.2} and since $m=2m'$ with $m'$ odd, we have
\[
S_1(m') = \tfrac{1}{2}m'\varphi(m') = \tfrac{1}{4}m\varphi(m),
\]
which gives \eqref{6.5}.
\end{proof}

Since this paper has been mostly about finite and infinite products, it is 
worth mentioning that multiplicative analogues of the sums $S_1(m)$ and 
$S_2(m)$ are special cases of the recently studied {\it Gauss factorials} 
\cite{CD}.

\section*{Acknowledgements}
The first author was supported in part by the Natural Sciences and 
Engineering Research Council of Canada. This work was completed during a visit by K. Dilcher
as invited Professor at Universit\'{e} Paris-Sud.

\end{document}